\documentclass[11pt,reqno]{amsart}
\usepackage{amssymb,amsfonts,amsmath,bm}
\usepackage{hyperref}
\usepackage{psfrag,graphicx,color}
\usepackage{mathrsfs}
\usepackage{graphics}
\usepackage{subfigure}
\usepackage{enumerate,enumitem}
\usepackage{pgfplots}
\usepackage{cite}
\usepackage{diagbox}
\usepackage{hyperref}

\allowdisplaybreaks

\usepackage{booktabs}
\usepackage{multirow}

\usepackage[margin=1.2in]{geometry}

\DeclareGraphicsExtensions{.eps}
\usepackage{amssymb,amsmath,graphicx,amsfonts,euscript,mathrsfs}
\usepackage{hyperref}

\usepackage{titlesec} 
\titleformat{\section}{\vskip10pt\large\bfseries}{\thesection.}{0.5em}{\centering\vspace{5pt}}
\titleformat{\subsection}{\vskip10pt\normalsize\bfseries}{\thesubsection.}{0.5em}{}

\newtheorem{theorem}{Theorem}[section]

\theoremstyle{definition}

\def\R{{\mathbb R}}

\def\d{{\mathrm d}}

\numberwithin{equation}{section}

\begin{document}

\title[]{A semi-implicit low-regularity integrator\\ for Navier--Stokes equations} 

\author[]{\,\,Buyang Li}
\address{\hspace*{-12pt}Buyang Li and Shu Ma: 
Department of Applied Mathematics, The Hong Kong Polytechnic University,
Hong Kong. {\it E-mail address}: {\tt buyang.li@polyu.edu.hk} {\rm \,and\,} {\tt maisie.ma@connect.polyu.hk}}

\author[]{\,\,Shu Ma}

\author[]{\,Katharina Schratz}
\address{\hspace*{-12pt}Katharina Schratz: 
Laboratoire Jacques-Louis Lions, Sorbonne Université,
Bureau : 16-26-315, 4 place Jussieu, Paris 5ème. {\it E-mail address}: {\tt katharina.schratz@sorbonne-universite.fr} }





\maketitle

\vspace{-10pt}
\begin{abstract}
A new type of low-regularity integrator is proposed for Navier--Stokes equations, coupled with a stabilized finite element method in space. Unlike the other low-regularity integrators for nonlinear dispersive equations, which are all fully explicit in time, the proposed method is semi-implicit in time in order to preserve the energy-decay structure of NS equations. First-order convergence of the proposed method is established independent of the viscosity coefficient $\mu$, under weaker regularity conditions than other existing numerical methods, including the semi-implicit Euler method and classical exponential integrators. 
Numerical results show that the proposed method is more accurate than the semi-implicit Euler method in the viscous case $\mu=O(1)$, and more accurate than the classical exponential integrator in the inviscid case $\mu\rightarrow 0$. 
\end{abstract}



\section{\bf Introduction}\label{sec:intr}


This article is concerned with the numerical solution of the initial and boundary value problem of the incompressible Navier--Stokes (NS) equations  
\begin{equation}
\label{pde}
\left \{
\begin{aligned} 
\partial_t u + u\cdot\nabla u - \mu \varDelta u
+\nabla p &= 0 && \mbox{in}\,\,\,  \varOmega\times (0,T] ,\\
\nabla\cdot u&=0&&\mbox{in}\,\,\,  \varOmega\times (0,T] ,\\
u&=u_0 && \mbox{at}\,\,\, \varOmega\times \{0\} , 
\end{aligned}
\right .
\end{equation}
in a bounded domain $\Omega\subset\R^d$, with $d\in\{2,3\}$, under appropriate boundary conditions. 

NS equations are the fundamental partial differential equations describing the motion of incompressible viscous fluids. They are widely used in fluid dynamics to model water and blood flows, air flow around a wing, and ocean currents. 
As the exact solution is not known in most applications, the numerical solution of NS equations plays a  central role. The development of accurate, stable numerical methods, together with their rigorous error analysis, is therefore crucial and of major practical importance to reliable description of NS equations.
Driven by the  immense spectrum of applications, many different numerical methods have been proposed  for solving NS equations. 

In the smooth setting, i.e., for smooth solutions and regular initial data, the numerical approximation of NS is nowadays in large parts well understood and sharp rigorous global error estimates could be established; see, e.g., \cite{Heywood-Rannacher-1990,ingram-2013-new,layton-2014-numerical,marion-1998-navier,Nochetto-Pyo-2004,shen-1992-error-1,shen-1992-error-2}. The optimal-order error estimates generally use the viscosity term to  {control} the nonlinear term, and therefore contain a viscosity-dependent constant {$c(\mu^{-1})$} in the error bound in addition to {certain Sobolev} norms of the exact solution. Note that in case of large viscosity $\mu \sim 1$ the solution of NS is regularised such that non-smooth initial data is not a big problem numerically. 
{In particular, the rigorous error analysis of semi- and full discretisations of NS equations with $H^1$ initial data can be found in \cite{Hill-Suli-2000} and \cite{he2008,he2010}, respectively.} 
This, however, drastically changes in case of  small viscosity $\mu\ll 1$, where no smoothing can be expected and classical viscosity-dependent $c(\mu^{-1})$ error bounds explode.
Error estimates of numerical methods for NS equations without using the viscosity term to bound the nonlinear term {could recently be established for smooth solutions, see} for example in \cite{ARJR-2016,Bermejo-Saavedra-2016,Uchiumi2019}. 
The analysis in these articles show that the classical finite difference methods in time, such as the semi-implicit Euler and backward differentiation formulae, typically requires the solution to satisfy $u\in  L^\infty(0,T;H^2(\Omega)^d)$ and $\partial_{tt}u\in L^2(0,T;L^2(\Omega)^d)$ for first-order convergence in time and space (when the error constants do not depend on the viscosity), where $d$ denotes the dimension of space. The condition $\partial_{tt}u\in L^2(0,T;L^2(\Omega)^d)$ actually requires $ u\in L^2(0,T;H^4(\Omega)^d)$ for the solution of NS equations, as one time derivative of the solution is related to two spatial derivatives of the solution. As a result, the classical finite difference methods in time requires 
$$ u\in H^2(0,T;L^2(\Omega)^d)\cap L^2(0,T;H^4(\Omega)^d)\hookrightarrow L^\infty(0,T;H^3(\Omega)^d)$$ for first-order convergence in time and space. 
The analysis in the current paper further shows that the classical exponential integrators for NS equations, such as the exponential Euler method, would also require $u\in L^\infty(0,T;H^3(\Omega)^d)$ for first-order convergence in time (if we require the error constant to be independent of the viscosity). 

The objective of this article is to develop a {new} low-regularity integrator for NS which { allows for} first-order convergence in time and space under a weaker regularity condition $u\in L^\infty(0,T;W^{2,d+\epsilon}(\Omega)^d)$, where $\epsilon$ can be arbitrarily small. {In particular,  we present a stabilization technique by utilizing the nonlinear convection term in NS equations and establish global error estimates independent of $\mu$ allowing for low-regularity approximations also in regimes of small viscosity $\mu \ll1$.}

Our new scheme also greatly extends previous works on low regularity integrators which mainly focus on semi-discretizations in time \cite{K2} and nonlinear dispersive equations, e.g., Schr\"odinger, Dirac and  Korteweg-de Vries \cite{BS20,HoS16,ORS19,ORSBourg,OS18,SWZ-MC21,WuZ,WuZ2}.  In this work we approach the  NS equations, and, for the first time couple the idea of   low regularity time discretisations with  a finite element based spatial discretisation. 
Note that fully discrete low regularity  integrators were so far restricted to  pseudo spectral methods for the spatial discretisation \cite{Li-Wu-2020} which are not suitable for problems posed on general  bounded domains. The latter are, however,  especially interesting in the context of NS flow problems. 
The numerical experiments in this article show that the proposed low regularity integrator for NS equations is much more accurate than the classical semi-implicit Euler method in the viscous case $\mu=O(1)$, and more accurate and robust than the classical exponential integrator in the inviscid case $\mu\rightarrow 0$. Therefore, the proposed method combines the advantages of the semi-implicit Euler method and classical exponential integrator in both viscous and inviscid cases. 

The rest of this article is organized as follows. 
In Section \ref{sec:main_results} we construct a low-regularity integrator for NS equations through analyzing and improving both the consistency and the stability of the classical exponential Euler method. 
We first present the construction of the method in the context of periodic boundary conditions and then extend it to the widely used no-slip boundary conditions in NS flow problems. 
The energy-decay property and and error estimates of the proposed low-regularity integrator are proved for semidiscretization in time. 
In Section \ref{section:FEM} we extend the low-regularity integrator to full discretization with a stabilized finite element method in space, and present error estimates for the fully discrete low-regularity integrator. 
Numerical examples are presented in Section \ref{section:numerical} to compare the performance of the proposed low-regularity integrator with the performance of both the semi-implicit Euler method and the exponential Euler method. 
Conclusions and remarks are presented in Section \ref{sec:conclusion}.

\section{The low-regularity integrator and its basic properties}\label{sec:main_results}

In this section, we present the construction of the low-regularity integrator by analyzing the dependence of the consistency errors on the regularity of the solution. The construction is presented first for NS equations under the periodic boundary condition in subsection \ref{subsection:pbd} and then extended to the no-slip boundary condition  in subsection \ref{subsection:no-slip}.

\subsection{Construction of the time-stepping method}\label{subsection:pbd}
In this subsection, we focus on NS equations on the $d$-dimensional torus $\Omega=[0,1]^d$ (under the periodic boundary condition). Through integration by parts it is straightforward to verify the following property of the divergence-free subspace 
$$
\dot L^2 :=\{v\in L^2(\Omega)^d:\nabla\cdot v=0\} . 
$$
If $v\in \dot L^2$ and $q\in H^1$ then
$$
(v,\nabla q) = 0 . 
$$
Let $P_X:L^2(\Omega)^d\rightarrow \dot L^2$ be the $L^2$-orthogonal projection onto the divergence-free subspace. By using the above orthogonality, it is straightforward to verify that 
\begin{align} \label{P_X-formula}
P_Xf=f-\nabla q , 
\end{align}
where $q$ is the solution (up to a constant) of the following PDE problem (under periodic boundary conditions):
\begin{align*}
\Delta q=\nabla\cdot f . 
\end{align*}

Let $A=P_X\Delta: H^2\rightarrow \dot L^2$. Then NS equations can be written as 
\begin{align}\label{NS-abstract}
\left\{\begin{aligned}
\partial_tu + P_X( u\cdot\nabla u) - \mu Au &= 0 &&\mbox{for}\,\,\, t\in(0,T] ,\\
u(0) &=u_0 . 
\end{aligned}\right. 
\end{align}
From the definition of $A$ and identity \eqref{P_X-formula}, it is easy to see that for $f\in H^2(\Omega)^d$ 
\begin{align}\label{APf1}
AP_Xf=P_X\Delta P_Xf=P_X\Delta f-P_X\nabla \Delta q=P_X\Delta f 
\quad\mbox{(since $P_X\nabla \eta\equiv 0$)}  .
\end{align}
Moreover, if $f\in \dot H^2 =\{v\in H^2(\Omega)^d:\nabla\cdot v=0\}$ then $\Delta f\in \dot L^2$ and therefore $P_X\Delta f=\Delta f$. As a result, the following identity holds: 
\begin{align}\label{Af1} 
Af=\Delta f \quad\mbox{for} \,\,\, f\in \dot H^2 . 
\end{align}

Let $0=t_0<t_1<\cdots<t_N=T$ be a partition of the time interval $[0,T]$ with stepsize $\tau_n=t_{n}-t_{n-1}$. 
According to Duhamel's formula, the solution of \eqref{NS-abstract} satisfies the following identity: 
\begin{align} \label{Duhamel}
u(t_n)=e^{\tau_n \mu A}u(t_{n-1})-\int_{t_{n-1}}^{t_n} e^{(t_n-s)\mu A} P_X(u(s)\cdot \nabla u(s))\d s \quad \mbox{for}\,\,\, n\ge 1 .
\end{align}
The classical exponential integrator (for example, the exponential Euler method) approximates $u(s)$ by $u(t_{n-1})$ in \eqref{Duhamel}. 
Since 
\begin{align}\label{u(s)-u(tn-1)}
u(s) = u(t_{n-1}) + \mu \int_{t_{n-1}}^s Au(\sigma)\d\sigma - \int_{t_{n-1}}^s P_X( u(\sigma)\cdot\nabla u(\sigma))\d\sigma , 
\end{align}
substituting this identity into \eqref{Duhamel} yields that 
\begin{align} \label{Duhamel-Euler}
u(t_n)
=&\, 
e^{\tau_n \mu A}u(t_{n-1})
-\int_{t_{n-1}}^{t_n} e^{(t_n-s)\mu A} P_X(u(t_{n-1})\cdot \nabla u(t_{n-1}))\d s 
+ R_n , 
\end{align}
where the remainder $R_n$ is given by 
\begin{align*}
R_n
=&
-\int_{t_{n-1}}^{t_{n}}
e^{(t_n-s)\mu A} 
P_X\big[u(s)\cdot\nabla u(s) - u(t_{n-1})\cdot \nabla u(t_{n-1}) \big]\d s  \\
=&
-\int_{t_{n-1}}^{t_{n}}
e^{(t_n-s)\mu A} 
P_X\big[(u(s)-u(t_{n-1}))\cdot\nabla u(s) \big]\d s  \\
&
-\int_{t_{n-1}}^{t_{n}}
e^{(t_n-s)\mu A} 
P_X\big[u(t_{n-1})\cdot\nabla (u(s)-u(t_{n-1})) \big]\d s .
\end{align*}
By using the expression of $u(s)-u(t_{n-1})$ in \eqref{u(s)-u(tn-1)}, one can obtain the following estimate: 
\begin{align}\label{estimate-Rn}
\|R_n\|_{L^2} 
\lesssim
&\, \tau_n  \|u(s)-u(t_{n-1})\|_{L^\infty(0,T;L^2)} \|\nabla u\|_{L^\infty(0,T;L^{\infty})} \notag\\
&\, + \tau_n \|u\|_{L^\infty(0,T;L^\infty)} \|\nabla(u(s)-u(t_{n-1}))\|_{L^\infty(0,T;L^2)} \notag\\
\lesssim
&\, \mu \tau_n^2  \|u\|_{L^\infty(0,T;H^{2})} \|u\|_{L^\infty(0,T;W^{1,\infty})}
+ \tau_n^2 \|u\cdot\nabla u\|_{L^\infty(0,T;L^2)}  \|u\|_{L^\infty(0,T;W^{1,\infty})} \notag\\
&\, + \mu \tau_n^2 \|u\|_{L^\infty(0,T;L^\infty)} \|u\|_{L^\infty(0,T;H^{3})}
+ \tau_n^2 \|u\|_{L^\infty(0,T;L^\infty)} \|u\cdot\nabla u\|_{L^\infty(0,T;H^1)} \notag\\
\lesssim
&\, \mu \tau_n^2  \|u\|_{L^\infty(0,T;H^{2})} \|u\|_{L^\infty(0,T;H^3)}
+ \tau_n^2 \|u\|_{L^\infty(0,T;H^{2})}^2  \|u\|_{L^\infty(0,T;H^3)} . 
\end{align}
This requires $u\in L^\infty(0,T;H^3)$ in order to have first-order convergence in time (with second-order local truncation error). 

In contrast,
the idea behind the low-regularity integrator recently developed in \cite{K2}  lies in iterating Duhamel's formula  \eqref{Duhamel}, i.e., approximating  $u(s)$ by $e^{(s-t_{n-1}) \mu A}u(t_{n-1})$ in \eqref{Duhamel} and utilizing the relation 
\begin{align} \label{Duhamel2}
u(s)=e^{(s-t_{n-1}) \mu A}u(t_{n-1})-\int_{t_{n-1}}^{s} e^{(t_n-\sigma)\mu A} P_X(u(\sigma)\cdot \nabla u(\sigma))\d \sigma . 
\end{align}
We then rewrite the corresponding temporal integral by
\begin{align} \label{exp-NS-u-dxu}
&\int_{t_{n-1}}^{t_n} e^{(t_n-s)\mu A} P_X(u(s)\cdot \nabla u(s))\d s \notag \\ 
&= 
\int_{t_{n-1}}^{t_n} e^{(t_n-s)\mu A} P_X(e^{(s-t_{n-1}) \mu A}u(t_{n-1})\cdot \nabla e^{(s-t_{n-1}) \mu A}u(t_{n-1}))\d s 
+ R_{n,1} .
\end{align}
Compared with the formula \eqref{u(s)-u(tn-1)} used in the classical exponential integrator, the relation \eqref{Duhamel2} does not contain the term $Au$. As a result, the remainder $R_{n,1}$ in \eqref{exp-NS-u-dxu} satisfies the following improved estimate: 
\begin{align}\label{estimate-Rn1}
\|R_{n,1}\|_{L^2}
\lesssim
 \tau_n^2 \|u\|_{L^\infty(0,T;L^\infty)}^2 \|u\|_{L^\infty(0,T;H^2)} 
+ \tau_n^2 \|u\|_{L^\infty(0,T;L^\infty)} \|u\|_{L^\infty(0,T;W^{1,4})}^2 ,
\end{align}
which does not contain the $H^3$ norms of $u$ that appear in \eqref{estimate-Rn}. 

By substituting \eqref{exp-NS-u-dxu} into \eqref{Duhamel}, we obtain 
\begin{align} \label{Duhamel-LRI}
u(t_n)
=&\, e^{\tau_n \mu A}u(t_{n-1}) -
\int_{t_{n-1}}^{t_n} e^{(t_n-s)\mu A} P_X(e^{(s-t_{n-1}) \mu A}u(t_{n-1})\cdot \nabla e^{(s-t_{n-1}) \mu A}u(t_{n-1}))\d s 
- R_{n,1} \notag\\
=&\, e^{\tau_n \mu A}u(t_{n-1}) 
-
\int_{t_{n-1}}^{t_n} g(s) \d s 
- R_{n,1} , 
\end{align}
where 
$$
g(s) = e^{(t_n-s)\mu A} P_X [v(s)\cdot\nabla v(s)]
\quad\mbox{with}\quad v(s)=e^{(s-t_{n-1}) \mu A}u(t_{n-1}) . 
$$
Then we consider a Taylor series of the function $g(s)$ at $s=t_n$. 
Since \eqref{APf1} and \eqref{Af1} imply that $AP_Xf=P_X\Delta f=\Delta P_Xf$, by using this relation with $f=v(s)\cdot\nabla v(s)$ (in the second equality below) we have 
\begin{align}\label{dg-periodic}
g'(s) 
& = - e^{(t_n-s)\mu A} \mu A P_X [v(s)\cdot\nabla v(s)]  \notag\\
&\quad\,  
+ e^{(t_n-s)\mu A}P_X [\mu A v(s)\cdot\nabla v(s) +  v(s)\cdot\nabla \mu A v(s)]  \notag\\
& = - \mu e^{(t_n-s)\mu A}  P_X\Delta  [v(s)\cdot\nabla v(s)]  \notag\\
&\quad\,  
+ \mu e^{(t_n-s)\mu A}P_X [ \Delta v(s)\cdot\nabla v(s) +  v(s)\cdot\nabla \Delta v(s)] \notag\\
& = - \mu e^{(t_n-s)\mu A}  P_X  \big[v(s)\cdot\nabla \Delta v(s) + \Delta v(s)\cdot\nabla v(s) 
+ \mbox{$\sum_j$}\partial_j v(s)\cdot \nabla\partial_j v(s) \big]  \notag\\
&\quad\,  
+ \mu e^{(t_n-s)\mu A}P_X [ \Delta v(s)\cdot\nabla v(s) +  v(s)\cdot\nabla \Delta v(s)] \notag\\ 
& = 
- \mu e^{(t_n-s)\mu A}  P_X  \big[ \mbox{$\sum_j$}\partial_j v(s)\cdot \nabla\partial_j v(s) \big] 
 .
\end{align}
Since $g(s) = g(t_n) - \int_{s}^{t_n} g'(\sigma)\d\sigma$ and 
\begin{align*}
&\|g'(\sigma)\|_{L^2} \lesssim \mu \|\nabla v(\sigma)\|_{L^q} \|\nabla^2 v(\sigma)\|_{L^p} \quad \mbox{when}\,\,\,\, \frac1p+\frac1q=\frac12 , \\
&W^{2,p}\hookrightarrow W^{1,q} \quad \mbox{when}\,\,\,\, 1=\frac{d}{p}-\frac{d}{q}\,\,\,\mbox{and}\,\,\, 1\le  p\le q<\infty, 
\end{align*}
by choosing $1\le p\le q< \infty$ satisfying $\frac1p+\frac1q=\frac12$ and $1=\frac{d}{p}-\frac{d}{q}=\frac{2d}{p}-\frac{d}{2}$ we obtain 
\begin{align}\label{estimate-dg}
&\|g'(\sigma)\|_{L^2} \lesssim \mu \|v(\sigma)\|_{W^{2,p}}^2 
\end{align}
with
\begin{align}\label{def-p}
p = 
\left\{
\begin{aligned}
&\frac{2d}{1+d/2} = \frac{12}{5} &&\mbox{if}\,\,\, d=3 \\
&2+\epsilon  &&\mbox{if}\,\,\, d=2 , 
\end{aligned}
\right.
\end{align}
where $\epsilon>0$ can be arbitrarily small. 
Therefore, the following result holds: 
\begin{align*}
\| g(s) - g(t_n) \|_{L^2} 
\lesssim \mu \tau_n \|v\|_{L^\infty(0,T;W^{2,p})}^2 
 \lesssim \mu \tau_n \|u\|_{L^\infty(0,T;W^{2,p})}^2 . 
\end{align*}
In view of this estimate, we can rewrite \eqref{Duhamel-LRI} as 
\begin{align} \label{Duhamel-LRI-2}
u(t_n)
=&\, e^{\tau_n \mu A}u(t_{n-1}) 
-
\int_{t_{n-1}}^{t_n} g(t_n) \d s 
- R_{n,1} - R_{n,2} , 
\end{align}
with a new remainder $R_{n,2}$ which has the following bound: 
\begin{align} \label{estimate-Rn2}
\|R_{n,2}\|_{L^2} \lesssim \mu\tau_n^2 \|u\|_{L^\infty(0,T;W^{2,p})}^2 .
\end{align}
Inserting the expression of $g(t_n)$ into \eqref{Duhamel-LRI-2}, we have 
\begin{align} \label{Duhamel-LRI-3}
u(t_n)
=&\, e^{\tau_n \mu A}u(t_{n-1}) 
-
\tau_n P_X[ e^{\tau_n\mu A}u(t_{n-1}) \cdot \nabla e^{\tau_n \mu A}u(t_{n-1})  ]
- R_{n,1} - R_{n,2} .
\end{align}

Dropping the remainders $R_{n,1}$ and $R_{n,2}$ in \eqref{Duhamel-LRI-3} would yield a fully explicit scheme
\begin{align} \label{explicit-scheme}
u_n
=&\, e^{\tau_n \mu A}u_{n-1}
-
\tau_n P_X[ e^{\tau_n\mu A}u_{n-1} \cdot\nabla e^{\tau_n \mu A}u_{n-1} ] . 
\end{align}
However, in the stability estimate the gradient on the right-hand side should be bounded by the smoothing property of the semigroup $e^{\tau_n \mu A}$, and this would yield a stability estimate which depend on $\mu^{-1}$. This would not be suitable for solving NS equations when the viscosity $\mu$ is small. 

In order to construct a low-regularity integrator which is stable for small $\mu$, we further approximate $\nabla e^{\tau_n \mu A}u(t_{n-1}) $ by $\nabla u(t_{n}) $, and rewrite \eqref{Duhamel-LRI-3} into 
\begin{align} \label{Duhamel-LRI-4}
u(t_n)
=&\, e^{\tau_n \mu A}u(t_{n-1}) 
-
\tau_n P_X[ e^{\tau_n\mu A}u(t_{n-1}) \cdot \nabla u(t_{n})  ]
- R_{n,1} - R_{n,2} - P_XR_{n,3} , 
\end{align}
with
\begin{align} \label{Rn3}
R_{n,3} 
=&\, \tau_n e^{\tau_n\mu A}u(t_{n-1}) \cdot \nabla [e^{\tau_n \mu A}u(t_{n-1})  - u(t_{n}) ] \notag\\
=&\, \tau_n e^{\tau_n\mu A}u(t_{n-1}) \cdot \nabla \int_{t_{n-1}}^{t_n} e^{(t_n-s)\mu A} P_X(u(s)\cdot \nabla u(s))\d s 
\qquad\mbox{(here \eqref{Duhamel} is used)} \notag\\
=&\, \tau_n [e^{\tau_n\mu A}u(t_{n-1})]_j \cdot  \int_{t_{n-1}}^{t_n} e^{(t_n-s)\mu A} P_X(u(s)\cdot \nabla \partial_j u(s))\d s \notag\\
&\, 
+\tau_n [e^{\tau_n\mu A}u(t_{n-1})]_j \cdot  \int_{t_{n-1}}^{t_n} e^{(t_n-s)\mu A} P_X(\partial_j  u(s)\cdot \nabla u(s))\d s , 
\end{align}
where we have used \eqref{Duhamel} in deriving the second to last inequality. 
The new remainder has the following bound: 
\begin{align} \label{estimate-Rn3}
\| R_{n,3} \|_{L^2}
\lesssim &\, 
\tau_n^2 \|e^{\tau_n\mu A}u(t_{n-1})\|_{L^\infty} 
\|u\|_{L^\infty(0,T;L^\infty)} \|u\|_{L^\infty(0,T;H^2)} \notag\\
&\, + 
\tau_n^2 \|e^{\tau_n\mu A}u(t_{n-1})\|_{L^\infty} \|u\|_{L^\infty(0,T;W^{1,4})}^2 \notag\\
\lesssim &\, 
\tau_n^2\|u\|_{L^\infty(0,T;H^2)}^3 . 
\end{align}
Hence, the remainders in \eqref{Duhamel-LRI-4} are bounded by $O(\tau_n^2)$ in the $L^2$ norm, i.e.,
\begin{align} \label{Rn123}
\| R_{n,1} \|_{L^2}
+\| R_{n,2} \|_{L^2}
+\| R_{n,3} \|_{L^2}
\lesssim \tau_n^2,
\end{align}
which only requires $u\in L^\infty(0,T;W^{2,p})$, where $p$ is defined in \eqref{def-p}. 

By dropping the remainders $R_{n,1}$, $R_{n,2}$ and  $P_XR_{n,3}$ in \eqref{Duhamel-LRI-4}, we obtain the following semi-implicit scheme for NS equations: 
\begin{align} \label{semi-implicit}
u_n + \tau_n P_X[ e^{\tau_n\mu A}u_{n-1} \cdot\nabla u_n ]
=&\, e^{\tau_n \mu A}u_{n-1} . 
\end{align}

\subsection{Extension to the no-slip boundary condition}
\label{subsection:no-slip}

If $\Omega$ is a bounded domain in $\R^d$ and NS equations are considered under the no-slip boundary condition, i.e., $u=0$ on $\partial\Omega$, then the definition of $\dot L^2$ should be replaced by 
$$
\dot L^2=\{v\in L^2(\Omega)^d: \nabla\cdot v=0,\,\,\, v\cdot\nu=0\,\,\,\mbox{on}\,\,\,\partial\Omega\} ,
$$ 
where $\nu$ denotes the unit outward normal vector on the boundary $\partial\Omega$. 
The $L^2$-orthogonal projection $P_X:L^2(\Omega)^d\rightarrow \dot L^2$ is given by 
\begin{align} \label{P_X-formula_D}
P_Xf=f-\nabla q , 
\end{align}
where $q$ is the solution (up to a constant) of the following elliptic boundary value problem:
\begin{align*}
\left\{\begin{aligned}
\Delta q=&\nabla\cdot f \\
\nabla q\cdot\nu = &f\cdot\nu .
\end{aligned}\right.
\end{align*}

Let $\dot H^2=\{v\in (H^1_0\times H^2)^d:\nabla\cdot v=0\}$ and $A=P_X\Delta: \dot H^2\rightarrow \dot L^2$. Then NS equations can be written as \eqref{NS-abstract}. 
Since $P_X\nabla q=0$ for $q\in H^1(\Omega)^d$, applying $A=P_X\Delta$ to \eqref{P_X-formula_D} yields 
\begin{align}\label{APf1-D}
AP_Xf&=P_X\Delta f \quad\mbox{for}\,\,\, f\in (H^1_0\times H^2)^d ,
\end{align}
which is the same as \eqref{APf1}. 
But \eqref{Af1} should be replaced by 
\begin{align}\label{Af1-D}
Av&=P_X\Delta v = \Delta v - \nabla r \quad\mbox{for} \,\,\, v\in \dot H^2 , 
\end{align}
where 
\begin{align}\label{PDE-r}
\left\{\begin{aligned}
\Delta r=&\nabla\cdot \Delta v \\
\nabla r\cdot\nu = &\Delta v\cdot\nu .
\end{aligned}\right.
\end{align}
In a bounded Lipschitz domain it is known that the solution of \eqref{PDE-r} satisfies the basic $W^{1,p}$ estimate for some sufficiently small number $\epsilon_*>0$  (see \cite[Theorem 2]{Jerison-Kenig-1989}): 
\begin{align}\label{estimate-r}
\|r\|_{W^{1,p}}\lesssim \|v\|_{W^{2,p}} \quad\mbox{for}\,\,\, 2\le p<3+\epsilon_* . 
\end{align}

The change from \eqref{Af1} to \eqref{Af1-D} causes the change of analysis in the local truncation errors in \eqref{dg-periodic}, i.e., 
\begin{align}\label{dg-Dirichlet1}
g'(s) 
& = - e^{(t_n-s)\mu A} \mu A P_X [v(s)\cdot\nabla v(s)]  \notag\\
&\quad\,  
+ e^{(t_n-s)\mu A}P_X [\mu A v(s)\cdot\nabla v(s) +  v(s)\cdot\nabla \mu A v(s)]  \notag\\
& = - \mu e^{(t_n-s)\mu A}  P_X\Delta  [v(s)\cdot\nabla v(s)]  \notag\\
&\quad\,  
+ \mu e^{(t_n-s)\mu A}P_X [ (\Delta v(s)-\nabla r)\cdot\nabla v(s) +  v(s)\cdot\nabla (\Delta v(s)-\nabla r)] \notag\\
& = - \mu e^{(t_n-s)\mu A}  P_X  \big[v(s)\cdot\nabla \Delta v(s) + \Delta v(s)\cdot\nabla v(s) 
+ \mbox{$\sum_j$}\partial_j v(s)\cdot \nabla\partial_j v(s) \big]  \notag\\
&\quad\,  
+ \mu e^{(t_n-s)\mu A}P_X [ \Delta v(s)\cdot\nabla v(s) +  v(s)\cdot\nabla \Delta v(s) ] \notag\\
&\quad\,  
- \mu e^{(t_n-s)\mu A}P_X [ \partial_j r\partial_j v(s) +  v_j(s)\partial_j \nabla r] \notag\\
& = - \mu e^{(t_n-s)\mu A}  P_X  \big[ \mbox{$\sum_j$}\partial_j v(s)\cdot \nabla\partial_j v(s) \big]  \notag\\
&\quad\,  
- \mu e^{(t_n-s)\mu A}P_X [ \partial_j r\partial_j v(s) +  v_j(s)\partial_j \nabla r] , 
\end{align}
where some additional terms involving $\nabla^2 r$ appears, compared with \eqref{dg-periodic}. 
Since $\|\nabla^2 r\|_{L^2}$ is equivalent to $\|v\|_{H^3}$, the additional term involving $\nabla^2 r$ is not desired. Fortunately, the projection operator $P_X$ in the last term of \eqref{dg-Dirichlet1} cancels this bad term, i.e., 
\begin{align}\label{dg-Dirichlet2}
\eqref{dg-Dirichlet1} 
& = - \mu e^{(t_n-s)\mu A}  P_X  \big[ \mbox{$\sum_j$}\partial_j v(s)\cdot \nabla\partial_j v(s) \big]  \notag\\
&\quad\,  
- \mu e^{(t_n-s)\mu A}P_X [ \partial_j r\cdot\partial_j v(s) - \nabla v_j(s)\partial_j  r] \notag\\
&\quad\,
- \mu e^{(t_n-s)\mu A}P_X [ \nabla (v_j(s)\cdot\partial_j  r) ]
 .
\end{align}
Since $P_X\nabla q=0$ for all $q\in H^1(\Omega)$, it follows that the last term of \eqref{dg-Dirichlet2} is zero. This implies that 
\begin{align}\label{dg-Dirichlet}
g'(s) 
& = - \mu e^{(t_n-s)\mu A}  P_X  \big[ \mbox{$\sum_j$}\partial_j v(s)\cdot \nabla\partial_j v(s) \big]  \notag\\
&\quad\,  
- \mu e^{(t_n-s)\mu A}P_X [ \partial_j r\cdot\partial_j v(s) - \nabla v_j(s)\partial_j  r] . 
\end{align}
If $2\le p\le q< \infty$, $\frac1p+\frac1q=\frac12$ and $p<3+\epsilon_*$, then  
\begin{align*}
\|g'(s)\|_{L^2} 
\lesssim &\,
\mu \|\nabla v(s)\|_{L^q} (\|\nabla^2 v(s)\|_{L^p}+\|\nabla r\|_{L^p}) \\
\lesssim &\,
\mu \|\nabla v(s)\|_{L^q} \|\nabla^2 v(s)\|_{L^p}\quad\mbox{(here \eqref{estimate-r} is used)}.
\end{align*}
Since 
\begin{align*}
&W^{2,p}\hookrightarrow W^{1,q} \quad \mbox{when}\,\,\,\, 1=\frac{d}{p}-\frac{d}{q}\,\,\,\mbox{and}\,\,\, 1\le  p\le q<\infty, 
\end{align*}
by choosing $2\le p\le q< \infty$ satisfying $\frac1p+\frac1q=\frac12$ and $1=\frac{d}{p}-\frac{d}{q}=\frac{2d}{p}-\frac{d}{2}$ we obtain 
\begin{align}\label{estimate-dg-D}
&\|g'(s)\|_{L^2} \lesssim \mu \|v(s)\|_{W^{2,p}}^2 
\end{align}
with
$$
p = 
\left\{
\begin{aligned}
&\frac{12}{5}  &&\mbox{if}\,\,\, d=3 \\
&2+\epsilon  &&\mbox{if}\,\,\, d=2 , 
\end{aligned}
\right.
$$
where $\epsilon>0$ can be arbitrarily small. Indeed, this choice of $p$ satisfies the condition $p<3+\epsilon_*$ required in \eqref{estimate-r}. 
Since the estimate \eqref{estimate-dg-D} we obtained here is the same as \eqref{estimate-dg}, the rest analysis would be the same as the periodic boundary condition and therefore omitted. 
In the end, we would obtain \eqref{Duhamel-LRI-4} under the no-slip boundary condition, with remainders $R_{n,1}$, $R_{n,2}$ and $R_{n,3}$ satisfying the same estimates as that under periodic boundary conditions.
By dropping the remainders we would obtain the same semi-implicit scheme \eqref{semi-implicit}.

\subsection{The energy-decay property}
The proposed semi-implicit low-regularity integrator in \eqref{semi-implicit} preserves the energy-decay structure of NS equations. 
This can be seen by testing \eqref{semi-implicit} with $u_n$. Then we have 
\begin{align} \label{energy-decay-1}
\|u_n\|_{L^2}^2 + \tau_n (e^{\tau_n\mu A}u_{n-1} \cdot\nabla u_n , u_n)
=&\, (e^{\tau_n \mu A}u_{n-1} , u_n) . 
\end{align}
Since $e^{\tau_n\mu A}u_{n-1}$ is divergence-free (the same as $u_{n-1}$), it follows from integration by parts that 
$$
(e^{\tau_n\mu A}u_{n-1} \cdot\nabla u_n , u_n) 
=(e^{\tau_n\mu A}u_{n-1},\nabla \frac12|u_n|^2 ) 
=-(\nabla\cdot (e^{\tau_n\mu A}u_{n-1}), \frac12|u_n|^2 ) 
= 0 .
$$
As a result, \eqref{energy-decay-1} reduces to 
\begin{align*} 
\|u_n\|_{L^2}^2 
= (e^{\tau_n \mu A}u_{n-1} , u_n) 
\le \|e^{\tau_n \mu A}u_{n-1}\|_{L^2}\|u_n\|_{L^2}
\le \|u_{n-1}\|_{L^2}\|u_n\|_{L^2} ,
\end{align*}
which implies that 
\begin{align} \label{energy-decay-2}
\|u_n\|_{L^2} 
\le&\, \|u_{n-1}\|_{L^2} . 
\end{align}
On the one hand, the energy-decay structure of the semi-implicit low-regularity integrator guarantees the energy boundedness of the numerical solution without requiring any regularity of the solution and initial data. On the other hand, this energy-decay structure also plays an important role in guaranteeing the convergence of numerical solutions when the solution has sufficient regularity, as reflected by the error analysis below.

\subsection{Error estimates}

\begin{theorem}\label{THM:error-time}
Consider the NS equations either in a torus $\Omega=[0,1]^d$ with periodic boundary condition or in a bounded domain $\Omega$ under the Dirichlet boundary condition, and assume that the solution of the NS equations has the following regularity: 
\begin{align}\label{reg-u}
u\in C([0,T];L^2(\Omega)^d)\cap L^\infty(0,T;W^{1,\infty}(\Omega)^d)\cap L^\infty(0,T;W^{2,p}(\Omega)^d) ,
\end{align} 
where $p$ is given by \eqref{def-p}. 
Then the numerical solution by the semi-implicit method \eqref{semi-implicit} has the following error bound:
\begin{align} 
\max_{1\le n\le N} \|e_n\|_{L^2} 
\lesssim \tau .
\end{align}
\end{theorem}

\begin{proof}
If the solution has regularity \eqref{reg-u} for some $p>d$, then $p$ is bigger than the value defined in \eqref{def-p} and therefore the regularity required in Section \ref{sec:main_results} is satisfied. 

Let $e_n=u_n -u(t_n)$ be the error function. The difference between \eqref{semi-implicit} and \eqref{Duhamel-LRI-4} yields the following error equation:  
\begin{align} \label{error_eq}
e_n + \tau_n P_X[ e^{\tau_n\mu A}u_{n-1} \cdot\nabla e_n ]
=&\, e^{\tau_n \mu A}e_{n-1} 
- \tau_n P_X[ e^{\tau_n\mu A}e_{n-1} \cdot\nabla u(t_n) ] \notag\\
&\,
+R_{n,1}+R_{n,2}+P_XR_{n,3}. 
\end{align}
Testing \eqref{error_eq} by $e_n$ and using the consistency error estimates in \eqref{Rn123}, we obtain 
\begin{align*} 
\|e_n\|_{L^2}^2 
=&\, 
(e^{\tau_n \mu A}e_{n-1} , e_n) 
- (\tau_n e^{\tau_n\mu A}e_{n-1} \cdot\nabla u(t_n) , e_n) 
+ (R_{n,1}+R_{n,2}+P_XR_{n,3} , e_n) \\
\le&\, 
\frac12\|e_{n-1}\|_{L^2}^2+\frac12\|e_{n}\|_{L^2}^2 
+C\tau_n\|\nabla u(t_n)\|_{L^\infty}  \|e_{n-1}\|_{L^2}\|e_{n}\|_{L^2} 
+C\tau_n^2\|e_{n}\|_{L^2} \\
\le&\, 
\frac12\|e_{n-1}\|_{L^2}^2+\frac12\|e_{n}\|_{L^2}^2 
+C\tau_n \|e_{n-1}\|_{L^2}^2 
+C\tau_n \|e_{n}\|_{L^2}^2+C\tau_n^3 .
\end{align*}
The second and fourth terms on the right-hand side can be absorbed by the left-hand side. 
Therefore, we have 
\begin{align*} 
(1-C\tau_n)\|e_n\|_{L^2}^2 
\le&\, 
(1+C\tau_n)\|e_{n-1}\|_{L^2}^2  
+C\tau_n^3 .
\end{align*}
For sufficiently small stepsize $\tau_n$ we can apply Gronwall's inequality. This yields 
\begin{align*} 
\max_{1\le n\le N} \|e_n\|_{L^2}^2 
\lesssim \tau^2 .
\end{align*}
This proves the desired error bound in Theorem \ref{THM:error-time}. 
\end{proof}

\section{Extension to fully discrete finite element methods}
\label{section:FEM}


In this section, we extend the low-regularity integrator to full discretization by using a finite element method with postprocessing at every time level. For simplicity we focus on the periodic boundary condition. 

We consider a conforming finite element subspace $X_h\times M_h\subset H^1(\Omega)^d\times L^2(\Omega)$ with the following two properties:
\begin{enumerate}
\item
The inf-sup condition:
$$
\|q^h\|_{L^2} \lesssim \sup_{\begin{subarray}{c}v^h\in X_h\\
v^h\neq 0\end{subarray}} \frac{(\nabla \cdot v^h,q^h)}{\|v^h\|_{H^1}} .
$$

\item
Approximation properties:
\begin{align*}
&\inf_{v^h\in X_h} (\|v-v^h\|_{L^2} + h \|v-v^h\|_{H^1}) \lesssim h^k\|v\|_{H^k} 
&&\mbox{for} \,\,\, v\in H^k(\Omega)^d\,\,\,\mbox{and} \,\,\, 1\le k\le 2 ,\\
&\inf_{q^h\in M_h} \|q-q^h\|_{L^2} \lesssim h^k\|q\|_{H^k} 
&&\mbox{for} \,\,\, q\in H^k(\Omega)\,\,\,\mbox{and} \,\,\, 0\le k\le 1 . 
\end{align*}

\end{enumerate}
Examples of such finite element spaces include the Taylor--Hood P$^k$-P$^{k-1}$ spaces (for $k\ge 2$) and the mini-element P$^{\rm 1b}$-P$^1$ space; see \cite{Arnold-Brezzi-Fortin-1984,BrennerScott2008,Brezzi-Falk-1991}. 

We define the discrete divergence-free subspace of $X_h$ by  
$$\dot X_h=\{v^h \in X_h: (\nabla\cdot v^h,q^h)=0\,\,\mbox{for all}\,\,\,q^h\in M_h\} ,$$ 
and then define the discrete Stokes operator $A_h:\dot X_h\rightarrow \dot X_h$ by 
$$
(A_hw^h,v^h)=-(\nabla w^h,\nabla v^h) \quad\forall\, w^h,v^h\in \dot X_h.
$$

We define $U_h$ to be the $H(\text{div}, \Omega)$-conforming Raviart--Thomas 
finite element spaces of order $1$, i.e., 
$$U_h := \{w \in H(\text{div}, \Omega): w|_{K} \in P_{1}(K)^{d} + xP_{1}(K)\,\,\,\mbox{for every triangle}\,\,\, K \} ,
$$ 
and define the divergence-free subspace of $U_h$ by
\begin{align}
&\dot U_h 
:=\{v_h\in U_h :
\mbox{$\nabla\cdot v_h = 0$ in $\Omega$}\} .
\label{FEM_space_post_processed_velocity}
\end{align}
Let $P_{\dot U_h}:L^2(\Omega)^d \rightarrow \dot U_h $ be the $L^{2}$-orthogonal projection, defined by 
\begin{align}\label{L2-Proj-RT}
\,\,\,(v-P_{\dot U_h}v,w_h)=0
\quad\forall\, w_h\in \dot U_h,\,\,\,\forall\,v\in L^2(\Omega)^d . 
\end{align}
If $v\in H^2(\Omega)^d$ is a divergence-free vector field then the following approximation result holds (see \cite[inequality (3.4)]{Li-Qiu-arXiv})
\begin{align}\label{L2-error-Proj-RT}
\|v-P_{\dot U_h}v\|_{L^2}\le Ch^2\|v\|_{H^2} . 
\end{align}

Note that the weak formulation of the time-stepping method in \eqref{semi-implicit} can be written as 
\begin{align} 
&(u_n,v) + ( \tau_n  e^{\tau_n\mu A}u_{n-1} \cdot\nabla u_n ,v)
+(p_n,\nabla \cdot v)  =(e^{\tau_n \mu A}u_{n-1} ,v) &&\forall\, v\in H^1(\Omega)^d , \label{semi-implicit-weak1}\\[5pt] 
&(\nabla\cdot u_n,q)= 0  &&\forall\, q\in L^2(\Omega) , \label{semi-implicit-weak2}
\end{align}
where $p_n$ is the function satisfying 
$$\tau_n P_X[e^{\tau_n\mu A}u_{n-1} \cdot\nabla u_n]=\tau_n   e^{\tau_n\mu A}u_{n-1} \cdot\nabla u_n-\nabla p_n . $$
By using the discrete Stokes operator $A_h$ and the projection operator $P_{\dot U_h}$ introduced in this section, we consider the following fully discrete finite element method for \eqref{semi-implicit-weak1}--\eqref{semi-implicit-weak2}: Find $(u_n^h,p_n^h)\in X_h\times M_h$ such that the following equations hold: 
\begin{align} 
&(u_n^h,v^h) 
+ \tau_n  ( [{ P_{\dot U_h} }e^{\tau_n\mu A_h}u_{n-1}^h] \cdot\nabla u_n^h ,v^h) 
+(p_n^h,\nabla \cdot v^h)
=(e^{\tau_n \mu A_h}u_{n-1}^h ,v^h) &&\forall\, v^h\in X_h , \label{FEM-weak1}\\
&(\nabla\cdot u_n^h,q^h)= 0  &&\forall\, q^h\in M_h . \label{FEM-weak2}
\end{align}

The presence of the postprocessing projection $P_{\dot U_h} $ is necessary for obtaining error estimates as well as preserving the energy-decay structure. 
In particular, since $\tau_n P_{\dot U_h} e^{\tau_n\mu A_h}u_{n-1}^h$ is divergence-free (due to the projection $P_{\dot U_h} $), it follows that 
$$
\tau_n ( [P_{\dot U_h} e^{\tau_n\mu A_h}u_{n-1}^h] \cdot\nabla u_n^h , u_n^h) = 0 .
$$
As a result, choosing $(v^h,q^h)=(u_n^h,p_n^h)$ in \eqref{FEM-weak1}--\eqref{FEM-weak2} yields 
\begin{align} 
&\|u_n^h\|_{L^2}^2 
=(e^{\tau_n \mu A_h}u_{n-1}^h ,u_n^h) 
\le
\|e^{\tau_n \mu A_h}u_{n-1}^h\|_{L^2}\|u_n^h\|_{L^2} 
\le
\|u_{n-1}^h\|_{L^2}\|u_n^h\|_{L^2} ,
\end{align}
which implies the following energy-decay inequality: 
\begin{align} 
&\|u_n^h\|_{L^2} 
\le
\|u_{n-1}^h\|_{L^2} . 
\end{align}

\begin{theorem}\label{THM:FD}
Consider the NS equations either in a torus $\Omega=[0,1]^d$ (with periodic boundary condition) and assume that the solution of the NS problem \eqref{pde} has the following regularity: 
\begin{align}\label{reg-u-fd}
u\in C([0,T];L^2(\Omega)^d)\cap L^\infty(0,T;W^{1,\infty}(\Omega)^d)\cap L^\infty(0,T;W^{2,p}(\Omega)^d) , 
\end{align} 
where $p$ is given by \eqref{def-p}. 
Then, under mesh size restriction $h\lesssim \tau_{\min}$ {\rm(}the smallest stepsize{\rm)}, the numerical solution given by the fully discrete method \eqref{FEM-weak1}--\eqref{FEM-weak2} has the following error bound:
\begin{align} 
\max_{1\le n\le N} \|u_n^h-u(t_n)\|_{L^2} 
\lesssim \tau .
\end{align}
\end{theorem}

\begin{proof}
By requiring the test function $v^h$to be in the discrete divergence-free subspace $\dot X_h$, 
the weak formulation \eqref{FEM-weak1}--\eqref{FEM-weak2} can be equivalently written as: Find $u_n^h\in \dot X_h$ such that 
%
%
\begin{align} 
&(u_n^h,v^h) 
+ \tau_n ( P_{\dot U_h}  [e^{\tau_n\mu A_h}u_{n-1}^h] \cdot\nabla u_n^h,v^h)
=(e^{\tau_n \mu A_h}u_{n-1}^h ,v^h) 
\quad \forall\, v^h\in \dot X_h . \label{FEM-div0}
\end{align}
The exact solution satisfies similar equations, i.e., 
\begin{align} 
&(P_{\dot X_h} u(t_n) , v^h) 
+ \tau_n (P_{\dot U_h} [e^{\tau_n\mu A_h}P_{\dot X_h} u(t_{n-1})] \cdot \nabla P_{\dot X_h} u(t_{n}) , v^h)  \notag\\
&= (e^{\tau_n \mu A_h}P_{\dot X_h} u(t_{n-1})  , v^h) 
- (R_{n,1} + R_{n,2} + P_XR_{n,3} ,v^h) \notag\\
&\quad\, 
-  (E_{n,1} + E_{n,2} + E_{n,3}, v^h)  &&\forall\, v^h\in \dot X_h , 
 \label{PDE-div0-n1} 
\end{align} 
where
\begin{align} 
E_{n,1} 
=&\, \tau_n \big[ e^{\tau_n\mu A}u(t_{n-1})-P_{\dot U_h} e^{\tau_n\mu A_h}P_{\dot X_h} u(t_{n-1}) \big] \cdot \nabla u(t_{n}) , \\
E_{n,2} 
=&\, \tau_n P_{\dot U_h} \big[ e^{\tau_n\mu A_h}P_{\dot X_h} u(t_{n-1})\big]  \cdot \nabla (u(t_{n})-P_{\dot X_h}u(t_{n})) , \\
E_{n,3} 
=&\, e^{\tau_n \mu A_h}P_{\dot X_h} u(t_{n-1}) - P_{\dot X_h} e^{\tau_n \mu A}u(t_{n-1}) .
\end{align} 

By using the triangle inequality we can decompose $E_{n,1}$ into two parts, i.e.,
\begin{align} \label{Estimate-En1}
\|E_{n,1}\|_{L^2}  
\le &\, \tau_n 
\big\| \big[e^{\tau_n\mu A}u(t_{n-1})-P_{\dot U_h} e^{\tau_n\mu A}u(t_{n-1}) \big]\cdot \nabla u(t_{n}) \big\|_{L^2} \\
&\,+ \tau_n\big\| P_{\dot U_h} \big[e^{\tau_n\mu A}u(t_{n-1})-e^{\tau_n\mu A_h}P_{\dot X_h} u(t_{n-1})\big]\cdot \nabla u(t_{n}) \big\|_{L^2} \notag \\
\le&\,
\tau_nh^2\|e^{\tau_n\mu A}u(t_{n-1})\|_{H^2} \|\nabla u(t_n)\|_{L^\infty}
+\tau_nh^2\|u(t_{n-1})\|_{H^2} \|\nabla u(t_n)\|_{L^\infty}, 
\end{align} 
where the first term on the right-hand side of \eqref{Estimate-En1} is estimated by using \eqref{L2-error-Proj-RT}, and the second term is estimated by using the standard $L^2$ error estimates of semidiscrete FEM for a linear parabolic equation with initial value $u(t_{n-1})$; see \cite[Theorem 3.1]{Thomee2006} (for the time-dependent Stokes equations the error estimation is the same).
The standard approximation property of the $L^2$ projection operator $P_{\dot X_h}$ implies that 
\begin{align} 
\|E_{n,2}\|_{L^2} 
&\lesssim \tau_n h \|u\|_{L^\infty(0,T;H^2)}^2 .
\end{align}
Again, the standard $L^2$ error estimates of semidiscrete FEM for linear parabolic equations with initial value $u(t_{n-1})$ in \cite[Theorem 3.1]{Thomee2006} implies that 
\begin{align} 
\|E_{n,3}\|_{L^2} 
&\lesssim h^2 \|u(t_{n-1})\|_{H^2} . 
\label{En3}
\end{align}
The three estimates above can be summarized as 
\begin{align} 
\|E_{n,1}\|_{L^2} +\|E_{n,2}\|_{L^2} +\|E_{n,3}\|_{L^2} 
&\lesssim \tau_nh + h^2 . 
\label{En123}
\end{align}
%


Let $e_n^h=u_n^h-P_{\dot X_h}u(t_n)$. 
Then the difference between \eqref{FEM-div0} and \eqref{PDE-div0-n1} yields the following error equation: 
\begin{align} 
&(e_n^h,v^h)  
+ \tau_n  ( P_{\dot U_h} \big[e^{\tau_n\mu A_h}u_{n-1}^h\big] \cdot\nabla e_n^h,v^h)
+ \tau_n  (P_{\dot U_h}\big[e^{\tau_n\mu A_h}e_{n-1}^h\big] \cdot\nabla P_{\dot X_h} u(t_n),v^h) 
\label{Error_Eq} \\
&
=(e^{\tau_n \mu A_h}e_{n-1}^h ,v^h) 
- (R_{n,1} + R_{n,2} + P_XR_{n,3} ,v^h) 
- (E_{n,1} + E_{n,2} + E_{n,3}, v^h)  
\quad \forall\, v^h\in \dot X_h . \notag 
\end{align}
By choosing $v^h=e_n^h$ in \eqref{Error_Eq}  and using the property (thanks to the projection $P_{\dot U_h}$ onto the divergence-free space $\dot U_h$) 
$$
( P_{\dot U_h} \big[e^{\tau_n\mu A_h}u_{n-1}^h\big] \cdot\nabla e_n^h,e_n^h) = 0 , 
$$
we obtain 
\begin{align} 
&
\|e_n^h\|_{L^2}^2  
+ \tau_n ( P_{\dot U_h}\big[e^{\tau_n\mu A_h}e_{n-1}^h\big] \cdot\nabla P_{\dot X_h} u(t_n),e_n^h) \notag\\
&
=(e^{\tau_n \mu A_h}e_{n-1}^h , e_n^h ) 
- (R_{n,1} + R_{n,2} + P_XR_{n,3} , e_n^h) 
- (E_{n,1} + E_{n,2} + E_{n,3}, e_n^h) . 
\end{align}
The rigth-hand side of the above inequality can be estimated by using the consistency error estimates in \eqref{Rn123} and \eqref{En123}. This yields 
\begin{align} 
&
\|e_n^h\|_{L^2}^2  
+ \tau_n ( P_{\dot U_h}\big[e^{\tau_n\mu A_h}e_{n-1}^h\big] \cdot\nabla P_{\dot X_h} u(t_n),e_n^h) \notag\\
&\le 
\frac12\|e_{n-1}^h\|_{L^2}^2+\frac12\|e_{n}^h\|_{L^2}^2 
+C\tau_n(\tau_n+h)\|e_{n}^h\|_{L^2} 
+Ch^2  \|e_{n}^h\|_{L^2} \notag\\
&\le 
\frac12\|e_{n-1}^h\|_{L^2}^2+\frac{1+\tau_n}{2}\|e_{n}^h\|_{L^2}^2 
+C[\tau_n(\tau_n^2+h^2) + h^4/\tau_n] \notag\\
&\le 
\frac12\|e_{n-1}^h\|_{L^2}^2+\frac{1+\tau_n}{2}\|e_{n}^h\|_{L^2}^2 
+C\tau_n^3 \qquad\mbox{when}\,\,\, h\lesssim \tau_n.
\end{align} 
The second term on the left-hand side of the above inequality can be estimated by 
\begin{align} 
|\tau_n ( P_{\dot U_h}\big[e^{\tau_n\mu A_h}e_{n-1}^h\big] \cdot\nabla P_{\dot X_h} u(t_n),e_n^h)| 
&\lesssim  
\tau_n \|e_{n-1}^h\|_{L^2} \|e_{n}^h\|_{L^2} .
\end{align} 
By combining the two inequalities above, we obtain 
\begin{align} 
(1-\tau_n) \|e_n^h\|_{L^2}^2  
&\le 
(1+C\tau_n)\|e_{n-1}^h\|_{L^2}^2  
+C\tau_n^3 \qquad\mbox{when}\,\,\, h\lesssim \tau_n.
\end{align} 
Then, iterating the inequality for $n=1,2,\dots,N$, we obtain the following error bound: 
\begin{align} 
\max_{1\le n\le N} \|e_n^h\|_{L^2}^2
&\lesssim 
\tau^2 \qquad\mbox{when}\,\,\, h\lesssim \tau_{\min} . 
\end{align} 
This completes the proof of Theorem \ref{THM:FD}.
\hfill \end{proof}

\section{Numerical experiments}\label{section:numerical}

%
%

In this section, we present numerical tests to support the theoretical analysis and to illustrate the advantages of the proposed method in comparison with the semi-implicit Euler method and classical exponential integrator (i.e., the exponential Euler method). 

We solve NS equations in the two-dimensional torus $[0,1]\times[0,1]$ under the periodic boundary condition by the proposed method, with initial value 
$$u^0=(u_{1}^0(x,y),u_{2}^0(x,y)) , $$ 
where  
\begin{align*}
u_{1}^0(x,y)&=m \pi \sin^{m}(\pi x)\sin^{m - 1}(\pi y)\cos(\pi y),\\
u_{2}^0(x,y)&=-m\pi \sin^{m}(\pi x)\cos(\pi x)\sin^{m -1}(\pi y).
\end{align*}
By choosing $m = 2.6$, the initial value satisfies $u^0\in H^{2+\epsilon}(\Omega)^2$ for $0<\epsilon<0.1$. 
Therefore, the initial value satisfies the conditions in Theorem \ref{THM:FD}. 
The algorithm in \cite{Lopez-Fernandez-2010} is used to evaluate the exponential operators in the low-regularity integrator and exponential Euler method. 

We present the time discretization errors $\|u_N^{(\tau)} - u_N^{(\tau/2)}\|_{L^2(\Omega)}$ of the numerical solutions at time $T = 1/8$ in Tables \ref{table_time_errors_M2_mu05_T1}--\ref{table_time_errors_M2_mu00001_T1} for several different $\mu$ (from $\mu=0.5$ to $\mu=10^{-4}$). The rate of convergence is computed based on the errors from the finest two mesh sizes, and $u_{\tau,h}$ denotes the numerical solution using stepsize $\tau$ and mesh size $h$. 
We have presented the time discretization errors using several different spatial mesh sizes to show that the spatial discretization errors is indeed negligibly small in observing the convergence rates in time. 
From Tables \ref{table_time_errors_M2_mu05_T1}--\ref{table_time_errors_M2_mu00001_T1} we see that the proposed low regularity integrator is more accurate than the semi-implicit Euler method when $\mu=O(1)$, and is equally accurate as the classical semi-implicit Euler method when $\mu\rightarrow 0$; 
at the same time, the proposed low regularity integrator is equally accurate as the classical exponential Euler method when $\mu=O(1)$, but is more accurate and robust than the classical exponential Euler method when $\mu\rightarrow 0$. 

We present the spatial discretization errors $\|u_N^{(h)} - u_N^{(h/2)}\|_{L^2(\Omega)}$ of the numerical solutions at $T = 1/8$ in Tables \ref{table_space_errors_M2_mu05_T1}--\ref{table_space_errors_M2_mu00001_T1}.  We have presented the spatial discretization errors using several different time stepsizes to show that the time discretization errors are indeed negligibly small in observing the convergence rates in space. 
From Tables \ref{table_space_errors_M2_mu05_T1}--\ref{table_space_errors_M2_mu00001_T1} we see that the proposed low regularity integrator has similar spatial discretization errors as the semi-implicit Euler method, but is more stable than the classical exponential Euler method when $\mu\rightarrow 0$, especially when the mesh size $h$ is small. 

Overall, the numerical results show that the proposed method is more accurate than the semi-implicit Euler method in the viscous case $\mu=O(1)$, and more robust than the classical exponential Euler method in the inviscid case $\mu\rightarrow 0$. Therefore, the proposed method combines the advantages of the semi-implicit Euler method and classical exponential integrator in both viscous and inviscid cases.

\begin{table}[h]\centering\small
\renewcommand\arraystretch{1.3}
\caption{Time discretization errors with $\mu = 0.5$ and $\tau = T/N$. }
\setlength{\tabcolsep}{3.1mm}{
\begin{tabular}{|c|l|cccc|c|}
\hline
$h$&  $N$&          $32$&        $64$&       $128$&       $256$&            Rate\\ 
\hline
\multirow{3}{*}{$1/16$}    
&  Exponential LRI&  3.4699e-06&  1.9484e-06&  1.0304e-06&  5.3118e-07&  $\approx$ 0.956\\
&        Semi-implicit Euler&  5.9899e-03&  2.9903e-03&  1.4939e-03&  7.4659e-04&  $\approx$ 1.001\\
& Exponential Euler&  8.5054e-06&  3.9837e-06&  2.2576e-06&  1.1120e-06&  $\approx$ 1.022\\
\hline
\multirow{3}{*}{$1/32$}    
&  Exponential LRI&  3.9203e-06&  2.1862e-06&  1.1592e-06&  5.9401e-07&  $\approx$ 0.965\\
&        Semi-implicit Euler&  6.0269e-03&  3.0090e-03&  1.5032e-03&  7.5129e-04&  $\approx$ 1.001\\
& Exponential Euler&  9.1872e-06&  4.3045e-06&  2.0812e-06&  1.0230e-06&  $\approx$ 1.025\\
\hline
\multirow{3}{*}{$1/64$}    
&  Exponential LRI&  4.0131e-06&  2.2432e-06&  1.1768e-06&  6.1235e-07&  $\approx$ 0.942\\
&        Semi-implicit Euler&  6.0357e-03&  3.0134e-03&  1.5055e-03&  7.5241e-04&  $\approx$ 1.001\\
& Exponential Euler&  9.3321e-06&  4.3740e-06&  2.1152e-06&  1.0399e-06&  $\approx$ 1.024\\
\hline
\end{tabular}
}
\label{table_time_errors_M2_mu05_T1}
\vspace{20pt}
%
\renewcommand\arraystretch{1.3}
\caption{Time discretization errors with $\mu = 0.1$ and $\tau = T/N$. }
\setlength{\tabcolsep}{3.1mm}{
\begin{tabular}{|c|l|cccc|c|}
\hline
$h$&  $N$&          $32$&        $64$&       $128$&       $256$&            Rate\\ 
\hline
\multirow{3}{*}{$1/16$}    
&  Exponential LRI&  3.1633e-04&  1.6263e-04&  8.2497e-05&  4.2000e-05&  $\approx$ 0.974\\
&        Semi-implicit Euler&  4.0897e-03&  2.0637e-03&  1.0366e-03&  5.1949e-04&  $\approx$ 0.997\\
& Exponential Euler&  5.0361e-04&  2.4379e-04&  1.1998e-04&  5.9526e-05&  $\approx$ 1.011\\
\hline
\multirow{3}{*}{$1/32$}    
&  Exponential LRI&  3.4307e-04&  1.7735e-04&  9.0155e-05&  4.3585e-05&  $\approx$ 1.049\\
&        Semi-implicit Euler&  4.0564e-03&  2.0470e-03&  1.0282e-03&  5.1530e-04&  $\approx$ 0.997\\
& Exponential Euler&  5.3366e-04&  2.5815e-04&  1.2701e-04&  6.3004e-05&  $\approx$ 1.012\\
\hline
\multirow{3}{*}{$1/64$}    
&  Exponential LRI&  3.4873e-04&  1.8048e-04&  9.1853e-05&  4.6309e-05&  $\approx$ 0.988\\
&        Semi-implicit Euler&  4.0477e-03&  2.0426e-03&  1.0261e-03&  5.1421e-04&  $\approx$ 0.997\\
& Exponential Euler&  5.4062e-04&  2.6148e-04&  1.2864e-04&  6.3811e-05&  $\approx$ 1.012\\
\hline
\end{tabular}
}
\label{table_time_errors_M2_mu01_T1}
%
\vspace{20pt}
%
\renewcommand\arraystretch{1.3}
\caption{Time discretization errors with $\mu = 0.01$ and $\tau = T/N$. }
\setlength{\tabcolsep}{3.1mm}{
\begin{tabular}{|c|l|cccc|c|}
\hline
$h$&  $N$&          $32$&        $64$&       $128$&       $256$&            Rate\\ 
\hline
\multirow{3}{*}{$1/16$}    
&  Exponential LRI&  3.2622e-03&  1.7456e-03&  9.0590e-04&  4.6192e-04&  $\approx$ 0.972\\
&        Semi-implicit Euler&  3.3815e-03&  1.8068e-03&  9.3673e-04&  4.7736e-04&  $\approx$ 0.973\\
& Exponential Euler&  3.0102e-03&  1.4021e-03&  6.8022e-04&  3.3534e-04&  $\approx$ 1.020\\
\hline
\multirow{3}{*}{$1/32$}    
&  Exponential LRI&  3.3153e-03&  1.7804e-03&  9.2140e-04&  4.7282e-04&  $\approx$ 0.963\\
&        Semi-implicit Euler&  3.4378e-03&  1.8411e-03&  9.5609e-04&  4.8771e-04&  $\approx$ 0.971\\
& Exponential Euler&  3.0755e-03&  1.4295e-03&  6.9242e-04&  3.4106e-04&  $\approx$ 1.022\\
\hline
\multirow{3}{*}{$1/64$}    
&  Exponential LRI&  3.3371e-03&  1.7930e-03&  9.3298e-04&  4.7647e-04&  $\approx$ 0.970\\
&        Semi-implicit Euler&  3.4582e-03&  1.8530e-03&  9.6258e-04&  4.9111e-04&  $\approx$ 0.971\\
& Exponential Euler&  3.0907e-03&  1.4363e-03&  6.9550e-04&  3.4253e-04&  $\approx$ 1.022\\
\hline
\end{tabular}
}
\label{table_time_errors_M2_mu001_T1}
\end{table}
%
%
\begin{table}[h]\centering\small
\renewcommand\arraystretch{1.3}
\caption{Time discretization errors with $\mu = 0.001$ and $\tau = T/N$. }
\setlength{\tabcolsep}{3.1mm}{
\begin{tabular}{|c|l|cccc|c|}
\hline
$h$&  $N$&          $32$&        $64$&       $128$&       $256$&            Rate\\ 
\hline
\multirow{3}{*}{$1/16$}    
&   Exponential LRI&  4.5100e-03&  2.4714e-03&  1.3045e-03&  6.7217e-04&  $\approx$ 0.957\\
&        Semi-implicit Euler&  4.5124e-03&  2.4725e-03&  1.3050e-03&  6.7249e-04&  $\approx$ 0.957\\
& Exponential Euler&  5.7820e-03&  2.2694e-03&  1.0388e-03&  5.0056e-04&  $\approx$ 1.053\\
\hline
\multirow{3}{*}{$1/32$}    
&  Exponential LRI&  4.4359e-03&  2.4191e-03&  1.2711e-03&  6.4915e-04&  $\approx$ 0.969\\
&        Semi-implicit Euler&  4.4414e-03&  2.4213e-03&  1.2720e-03&  6.5340e-04&  $\approx$ 0.961\\
& Exponential Euler&  8.9174e-02&  3.7611e-03&  9.9444e-04&  4.6723e-04&  $\approx$ 1.090\\
\hline
\multirow{3}{*}{$1/64$}    
&  Exponential LRI&  4.4536e-03&  2.4283e-03&  1.2754e-03&  6.5472e-04&  $\approx$ 0.962\\
&        Semi-implicit Euler&  4.4595e-03&  2.4308e-03&  1.2764e-03&  6.5523e-04&  $\approx$ 0.962\\
& Exponential Euler&         NAN&  4.7520e+06&  2.3803e-03&  4.6110e-04&  $\approx$ 2.368\\
\hline
\end{tabular}
}
\label{table_time_errors_M2_mu0001_T1}
%
\vspace{20pt}
%
\renewcommand\arraystretch{1.3}
\caption{Time discretization errors with $\mu = 0.0001$ and $\tau = T/N$. }
\setlength{\tabcolsep}{3.1mm}{
\begin{tabular}{|c|l|cccc|c|}
\hline
$h$&  $N$&          $32$&        $64$&       $128$&       $256$&            Rate\\ 
\hline
\multirow{3}{*}{$1/16$}    
&  Exponential LRI&  4.7418e-03&  2.6184e-03&  1.3904e-03&  7.1922e-04&  $\approx$ 0.951\\
&        Semi-implicit Euler&  4.7418e-03&  2.6185e-03&  1.3905e-03&  7.1930e-04&  $\approx$ 0.951\\
& Exponential Euler&  6.8225e-03&  2.5633e-03&  1.1510e-03&  5.5009e-04&  $\approx$ 1.065\\
\hline
\multirow{3}{*}{$1/32$}    
&  Exponential LRI&  4.6045e-03&  2.5219e-03&  1.3372e-03&  6.8424e-04&  $\approx$ 0.967\\
&        Semi-implicit Euler&  4.6049e-03&  2.5221e-03&  1.3318e-03&  6.8772e-04&  $\approx$ 0.954\\
& Exponential Euler&  2.4049e-01&  9.8285e-03&  1.3797e-03&  5.4065e-04&  $\approx$ 1.352\\
\hline
\multirow{3}{*}{$1/64$}    
&  Exponential LRI&  4.6126e-03&  2.5202e-03&  1.3256e-03&  6.8127e-04&  $\approx$ 0.960\\
&        Semi-implicit Euler&  4.6129e-03&  2.5204e-03&  1.3256e-03&  6.8131e-04&  $\approx$ 0.960\\
& Exponential Euler&         NAN&         NAN&  2.4302e-01&  1.8788e+00&  $\approx$ -2.951\\
\hline
\end{tabular}
}
\label{table_time_errors_M2_mu00001_T1}

\vspace{20pt}
\renewcommand\arraystretch{1.3}
\caption{Spatial discretization errors with $\mu = 0.5$. }
\begin{tabular}{|c|l|cccc|c|}
\hline
$\tau$&  $h$&      $1/8$&      $1/16$&      $1/32$&      $1/64$&            Rate\\ 
\hline
\multirow{3}{*}{$1/128$}    
&  Exponential LRI&  9.5008e-03&  2.3037e-03&  5.7092e-04&  1.5248e-04&  $\approx$ 1.905\\
&        Semi-implicit Euler&  1.0267e-02&  2.4771e-03&  6.1243e-04&  1.6382e-04&  $\approx$ 1.902\\
& Exponential Euler&  9.5102e-03&  2.3038e-03&  5.7020e-04&  1.5239e-04&  $\approx$ 1.904\\
\hline
\multirow{3}{*}{$1/256$}    
&  Exponential LRI&  9.5040e-03&  2.3049e-03&  5.7124e-04&  1.5259e-04&  $\approx$ 1.904\\
&        Semi-implicit Euler&  9.8841e-03&  2.3913e-03&  5.9192e-04&  1.5825e-04&  $\approx$ 1.903\\
& Exponential Euler&  9.5023e-03&  2.3029e-03&  5.7005e-04&  1.5235e-04&  $\approx$ 1.904\\
\hline
\multirow{3}{*}{$1/512$}    
&  Exponential LRI&  9.5040e-03&  2.3049e-03&  5.7124e-04&  1.5259e-04&  $\approx$ 1.904\\
&        Semi-implicit Euler&  9.6941e-03&  2.3484e-03&  5.8162e-04&  1.5544e-04&  $\approx$ 1.904\\
& Exponential Euler&  9.5012e-03&  2.3025e-03&  5.6998e-04&  1.5233e-04&  $\approx$ 1.904\\
\hline
\end{tabular}
\label{table_space_errors_M2_mu05_T1}
\end{table}

\begin{table}[h]\centering\small
%
\renewcommand\arraystretch{1.3}
\caption{Spatial discretization errors with $\mu = 0.01$. }
\setlength{\tabcolsep}{3.1mm}{
\begin{tabular}{|c|l|cccc|c|}
\hline
$\tau$&  $h$&      $1/8$&      $1/16$&      $1/32$&      $1/64$&            Rate\\ 
\hline
\multirow{3}{*}{$1/128$}    
&  Exponential LRI&  9.6583e-02&  1.5251e-02&  3.5192e-03&  8.8885e-04&  $\approx$ 1.985\\
&        Semi-implicit Euler&  8.9834e-02&  1.5103e-02&  3.4588e-03&  8.7148e-04&  $\approx$ 1.989\\
& Exponential Euler&  1.0360e-01&  1.6324e-02&  5.9873e-03&  2.9067e-03&  $\approx$ 1.043\\
\hline
\multirow{3}{*}{$1/256$}    
&  Exponential LRI&  9.7182e-02&  1.5293e-02&  3.5197e-03&  8.8697e-04&  $\approx$ 1.989\\
&        Semi-implicit Euler&  9.0903e-02&  1.5185e-02&  3.4737e-03&  8.7495e-04&  $\approx$ 1.989\\
& Exponential Euler&  1.0118e-01&  1.5469e-02&  3.5075e-03&  8.8679e-04&  $\approx$ 1.984\\
\hline
\multirow{3}{*}{$1/512$}    
&  Exponential LRI&  9.7624e-02&  1.5328e-02&  3.5237e-03&  8.8710e-04&  $\approx$ 1.990\\
&        Semi-implicit Euler&  9.1564e-02&  1.5240e-02&  3.4831e-03&  8.7711e-04&  $\approx$ 1.990\\
& Exponential Euler&  1.0028e-01&  1.5363e-02&  3.4891e-03&  8.7811e-04&  $\approx$ 1.990\\
\hline
\end{tabular}
}
\label{table_space_errors_M2_mu001_T1}

\vspace{20pt}
\caption{Spatial discretization errors with $\mu = 0.001$. }
\setlength{\tabcolsep}{3.1mm}{
\begin{tabular}{|c|l|cccc|c|}
\hline
$\tau$&  $h$&      $1/8$&      $1/16$&      $1/32$&      $1/64$&            Rate\\ 
\hline
\multirow{3}{*}{$1/128$}    
&  Exponential LRI&  1.5375e-01&  2.0932e-02&  4.3894e-03&  1.0286e-03&  $\approx$    2.093\\
&        Semi-implicit Euler&  1.4026e-01&  2.0292e-02&  4.2709e-03&  1.0018e-03&  $\approx$    2.092\\
& Exponential Euler&  1.8035e-01&  8.3996e-01&  5.0764e+08&  1.4240e+45&  $\approx$ -121.1\\
\hline
\multirow{3}{*}{$1/256$}    
&  Exponential LRI&  1.5710e-01&  2.1212e-02&  4.3977e-03&  1.0289e-03&  $\approx$ 2.096\\
&        Semi-implicit Euler&  1.4341e-01&  2.0587e-02&  4.2987e-03&  1.0072e-03&  $\approx$ 2.094\\
& Exponential Euler&  1.7388e-01&  8.9803e-02&         NAN&         NAN&    NAN\\
\hline
\multirow{3}{*}{$1/512$}    
&  Exponential LRI&  1.5913e-01&  2.1439e-02&  4.4077e-03&  1.0301e-03&  $\approx$ 2.097\\
&        Semi-implicit Euler&  1.4530e-01&  2.0806e-02&  4.3173e-03&  1.0108e-03&  $\approx$ 2.095\\
& Exponential Euler&  1.7192e-01&  2.3079e-02&  4.7352e+06&         NAN&  NAN  \\
\hline
\end{tabular}
}
\label{table_space_errors_M2_mu0001_T1}
%

\vspace{20pt}

\renewcommand\arraystretch{1.3}
\caption{Spatial discretization errors with $\mu = 0.0001$. }
\setlength{\tabcolsep}{3.1mm}{
\begin{tabular}{|c|l|cccc|c|}
\hline
$\tau$&  $h$&      $1/8$&      $1/16$&      $1/32$&      $1/64$&            Rate\\ 
\hline
\multirow{3}{*}{$1/128$}    
&  Exponential LRI&  1.6435e-01&  2.3918e-02&  5.7058e-03&  1.2805e-03&  $\approx$ 2.156\\
&        Semi-implicit Euler&  1.5036e-01&  2.3149e-02&  5.4845e-03&  1.2467e-03&  $\approx$ 2.137\\
& Exponential Euler&  1.9533e-01&  1.8706e+00&  4.9113e+29&         NAN&  $\approx$   NAN\\
\hline
\multirow{3}{*}{$1/256$}    
&  Exponential LRI&  1.6829e-01&  2.4376e-02&  5.7896e-03&  1.2914e-03&  $\approx$ 2.165\\
&        Semi-implicit Euler&  1.5378e-01&  2.3539e-02&  5.5662e-03&  1.2611e-03&  $\approx$ 2.142\\
& Exponential Euler&  1.8786e-01&  2.3386e-01&         NAN&        NAN&  NAN\\
\hline
\multirow{3}{*}{$1/512$}    
&  Exponential LRI&  1.7064e-01&  2.4733e-02&  5.8552e-03&  1.2989e-03&  $\approx$ 2.172\\
&        Semi-implicit Euler&  1.5583e-01&  2.3834e-02&  5.6235e-03&  1.2703e-03&  $\approx$ 2.146\\
& Exponential Euler&  1.8574e-01&  2.8885e-02&         NAN&         NAN&  NAN\\
\hline
\end{tabular}
}
\label{table_space_errors_M2_mu00001_T1}
\end{table}

\clearpage
\newpage

\section{Conclusions}\label{sec:conclusion}
We have proposed a semi-implicit fully discrete low-regularity integrator for NS equations under both periodic and Dirichlet boundary condition. This is for the first time a low-regularity integrator is coupled with a finite element method in space. The proposed method can be shown to have first-order convergence under weaker regularity conditions than the semi-implicit Euler method and classical exponential integrators. 
Under periodic boundary conditions, the numerical results show that the proposed method combines the advantages of the semi-implicit Euler method and classical exponential integrator in both viscous and inviscid cases. In particular, the proposed method is more accurate than the semi-implicit Euler method in the viscous case $\mu=O(1)$, and more robust than the classical exponential Euler method in the inviscid case $\mu\rightarrow 0$. 

Although the proposed low-regularity integrator has successfully weakened the regularity condition for first-order convergence and has improved the accuracy of classical methods in both viscous and inviscid cases under the periodic boundary conditions, 
the current version of low-regularity integrator still cannot resolve the boundary layer effect under the Dirichlet boundary condition in the inviscid case $\mu\rightarrow 0$. 
The construction of a low-regularity integrator which can resolve the boundary layer effect under the Dirichlet boundary condition in the inviscid case $\mu\rightarrow 0$ is an interesting and challenging task.

\subsection*{Acknowledgements}

{\small
K. Schratz has received funding from the European Research Council (ERC) under the European Union’s Horizon 2020 research and innovation programme (grant agreement No. 850941). 
The work of B. Li and S. Ma are supported in part by the internal grants ZZKQ and ZZKK of The Hong Kong Polytechnic University. 
}

\bibliographystyle{abbrv}
\bibliography{NS}

\end{document}